\documentclass[12pt,reqno]{amsart}
\usepackage{amssymb,amscd}
\usepackage[alphabetic]{amsrefs}
\usepackage[all]{xy}
\usepackage[headings]{fullpage}
\usepackage[T1]{fontenc}
\usepackage{libertine}
\usepackage{stackengine}

\numberwithin{equation}{section}

\theoremstyle{plain}
\newtheorem{thm}[subsection]{Theorem}

\newtheorem{prop}[subsection]{Proposition}
\newtheorem{propss}[subsubsection]{Proposition}
\newtheorem{lemma}[subsection]{Lemma}
\newtheorem{lemmass}[subsubsection]{Lemma}
\newtheorem{cor}[subsection]{Corollary}

\theoremstyle{definition}

\theoremstyle{remark}
\newtheorem{rem}[subsection]{Remark}

\numberwithin{equation}{section}

\usepackage[OT2,T1]{fontenc}
\DeclareSymbolFont{cyrletters}{OT2}{wncyr}{m}{n}
\DeclareMathSymbol{\sha}{\mathalpha}{cyrletters}{"58}

\renewcommand{\AA}{\mathcal{A}}

\newcommand{\CC}{\mathcal{C}}

\newcommand{\XX}{\mathcal{X}}

\newcommand{\GG}{{\mathcal{G}}}
\newcommand{\HH}{{\mathcal{H}}}

\newcommand{\FF}{\mathcal{F}}
\newcommand{\EE}{\mathcal{E}}

\newcommand{\LL}{\mathcal{L}}

\newcommand{\OO}{\mathcal{O}}

\newcommand{\Z}{\mathbb{Z}}

\newcommand{\Q}{\mathbb{Q}}
\newcommand{\R}{\mathbb{R}}
\newcommand{\C}{\mathbb{C}}
\newcommand{\A}{\mathbb{A}}
\renewcommand{\P}{\mathbb{P}}

\newcommand{\tensor}{\otimes}

\newcommand{\nodiv}{\not|}
\def\nodiv{\mathrel{\mathchoice{\not|}{\not|}{\kern-.2em\not\kern.2em|}
{\kern-.2em\not\kern.2em|}}}

\newcommand{\PSL}{\mathrm{PSL}}
\newcommand{\SL}{\mathrm{SL}}

\newcommand{\psmat}[1]{\bigl(\begin{smallmatrix}#1\end{smallmatrix}\bigr)}
\newcommand{\pmat}[1]{\begin{pmatrix}#1\end{pmatrix}}

\DeclareMathOperator{\im}{Im}

\DeclareMathOperator{\ord}{ord}

\DeclareMathOperator{\aut}{Aut}

\def\clap#1{\hbox to 0pt{\hss#1\hss}}

\begin{document}
\title[Bounding tangencies]{Bounding tangencies of sections on elliptic surfaces}

\author{Douglas Ulmer}
\address{Department of Mathematics \\ University of Arizona
  \\ Tucson, AZ~~85721 USA}
\email{ulmer@math.arizona.edu}

\author{Giancarlo Urz\'ua}
\address{Facultad de Matem\'aticas \\ Pontificia Universidad
  Cat\'olica de Chile \\ Santiago, Chile}
\email{urzua@mat.uc.cl}

\date{\today}


\begin{abstract}
  Given an elliptic surface $\mathcal{E}\to\mathcal{C}$ over a field
  $k$ of characteristic zero equipped with zero section $O$ and
  another section $P$ of infinite order, we give a simple and explicit
  upper bound on the number of points where $O$ is tangent to a
  multiple of $P$.
\end{abstract}

\maketitle

\section{Introduction}
Let $k=\C$ be the complex numbers\footnote{We will work over $\C$ for
  simplicity.  By a standard reduction given in \cite{UUpp19}, our
  results also hold when $k$ is any field of characteristic zero.},
let $\CC$ be an irreducible, smooth, projective curve of genus $g$
over $k$, and let $\pi:\EE\to\CC$ be a Jacobian elliptic surface over
$\CC$, i.e., an elliptic surface equipped with a section of $\pi$
denoted $O:\CC\to\EE$ which will play the role of a zero section.  Let
$P:\CC\to\EE$ be another section of $\pi$ which is of infinite order
in the group law with $O$ as origin.

Write $\EE[n]$ for the union of the points of order $n$ in each fiber
of $\pi$. It is known that $\EE[n]$ is a smooth, locally closed subset
of $\EE$ which is quasi-finite over $\CC$ of generic degree $n^2$ (See
\cite[Sections 2.1 and 2.2]{UUpp19} for more details.)  In
\cite{UUpp19}, we proved that the set
\begin{align*}
T_{tor}&:=\bigcup_{n>0}\left\{t\in\CC\left|nP\text{ is tangent to
        }O\text{ over }t\right.\right\}\\
&\phantom{:}=\bigcup_{n>0}\left\{t\in\CC\left|P\text{ is tangent to
        }\EE[n]\text{ over }t\right.\right\}
\end{align*}
is finite.  Our goal in this paper is to give an explicit upper bound for
$|T_{tor}|$, the cardinality of $T_{tor}$.  

We say that $\EE$ is \emph{constant} if there is an elliptic curve $E$
over $k$ such that $\EE\cong\CC\times_kE$ and $\pi$ is the projection
to $\CC$.  If $\EE$ is constant, we say $P$ is \emph{constant} if there is
point $p\in E$ such that $P(t)=(t,p)$ for all $t\in\CC$.
Let  $\delta$ be the number of singular fibers of $\pi:\EE\to\CC$,
let $\omega=O^*(\Omega^1_{\EE/\CC})$, and let $d=\deg(\omega)$.  Since
the discriminant of a Weierstrass model defines a non-vanishing section of
$\omega^{\tensor12}$, we have $d\ge0$.

\begin{thm}\label{thm:main}
  Suppose that $\EE$ is not constant, or that $\EE$ is constant and
  $P$ is not constant.  Then 
\[\left|T_{tor}\right|\le 2g-2-d+\delta.\]
\end{thm}
This is proved as Corollary~\ref{cor:bound} below.
In fact, we will prove a more precise result
(Theorem~\ref{thm:total-int}) which gives an \emph{exact
  formula} for the cardinality, with multiplicities, of a more general
set of tangencies.

\subsection{The constant case}
The constant case of our result is very transparent and gives a hint
of how to proceed in general, so we discuss it here.  Suppose that
$\EE\cong\CC\times E$ is constant.  Then a section $P:\CC\to\EE$ may
be identified with a morphism $f:\CC\to E$, and $P$ is constant if and
only if $f$ is constant.  We assume that $P$ is non-constant.  The
torsion subset $\EE[n]$ consists of the $n^2$ constant sections
$\CC\times\{p\}$ where $p$ is an $n$-torsion point of $E$.  It is of
interest to consider tangencies with general constant sections
$\CC\times\{p\}$ for any $p\in E$.  Let
\[T_{const}:=\bigcup_{p\in E}\left\{t\in\CC\left|P\text{ is tangent to
        }\CC\times\{p\}\text{ over }t\right.\right\}\]
i.e., the set of points of $\CC$ where $P$ is tangent to a constant
section.  Obviously $T_{tor}\subset T_{const}$.  

To take into account multiplicities, suppose $P(t)=(t,p)$ and let
$I(P,t)$ be the intersection number of $\CC\times\{p\}$ and $P$ at
$(t,p)$.  By definition, $I(P,t)\ge1$ and it is $\ge2$ if and only if
$P$ is tangent to $\CC\times\{p\}$ over $t$.  On the other hand, it is
clear that $I(P,t)$ is $e_t(f)$, the ramification index of $f$ at $t$.

Let $\eta_P$ be the pull-back under $f$ of a non-zero invariant
differential on $E$.  Since $f$ is non-constant, $\eta_P$ is a non-zero
section of $\Omega^1_\CC$, and the order of vanishing of $\eta_P$ at $t$
is 
\[\ord_t(\eta_P)=e_f(t)-1=I(P,t)-1.\]  
Thus we have
\[\left|T_{tor}\right|\le\left|T_{const}\right|
\le\sum_{t\in\CC}\left(I(P,t)-1\right)
=\sum_{t\in\CC}\ord_t(\eta_P)=2g-2.\]
Since $d=\delta=0$ when $\EE$ is constant, this proves
the Theorem~\ref{thm:main} in the constant case.

\subsection{Sketch of the general case}
In the general case, we will define a ``Betti foliation'' on an open
subset of $\EE$ which generalizes the foliation of $\CC\times E$ by the
leaves $\CC\times\{p\}$ and which has the subsets $\EE[n]$ among its closed
leaves.  This leads to a set of tangencies $T_{Betti}\subset\CC$ with
$T_{tor}\subset T_{Betti}$ and intersection multiplicities $I(P,t)$
which measure the order of contact between $P$ and the Betti
foliation.  We will also define a certain twisted \emph{real-analytic}
1-form $\eta_P$ on an open subset of $\CC$ whose local indices
$J(\eta_P,t)$ satisfy $J(\eta_P,t)=I(P,t)-1$ at all places $t$ of good
reduction.  Summing over all points of $\CC$ will lead to a formula
\[\sum_{t\in\CC}\left(I(P,t)-1\right)=\sum_{t\in\CC}J(\eta_P,t)=2g-2-d,\]
and taking into account
what happens at the bad fibers leads to the upper bound
\[\left|T_{tor}\right|\le\left|T_{Betti}\right|\le2g-2-d+\delta.\] 

A trivialization essentially equivalent to the Betti foliation
was used in the first version of \cite{UUpp19}, and we later adopted
the Betti terminology, following \cite{CorvajaMasserZannier18}.
The form $\eta_P$ appears implicitly in the first version of
\cite{UUpp19}.  A more general version of it is discussed at some
length in \cite[\S4]{ACZpp18}, and their account inspired our use of it
here to count tangencies.  The finiteness of $T_{tor}$ was proved
independently in \cite{CorvajaDemeioMasserZannierpp}.

\subsection{Plan of the paper}
In Section~\ref{s:analytic} we review certain aspects of Kodaira's
construction of $\EE$ as an analytic surface.  In Sections~\ref{s:Betti} 
and \ref{s:intersections}, we define the Betti foliation and local
intersection numbers $I(P,t)$ measuring the order of contact between a
section $P$ and the Betti foliation.  In Section~\ref{s:1-form}, we
attach to $P$ a real-analytic section $\eta_P$ of
$\Omega^1_\CC\tensor\omega^{-1}$ over a Zariski open subset of $\CC$
and with isolated zeroes, define local
indices $J(\eta_P,t)$, and calculate their sum.  In
Section~\ref{s:zeros}, we relate the local indices $I(P,t)$ and
$J(\eta_P,t)$.  This leads to the proof, in Section~\ref{s:proofs}, of
the main theorem.  Finally, in Section~\ref{s:examples} we give
examples illustrating edges cases and the sharpness of the main
theorem, and we give an application to heights of integral points on
elliptic curves over function fields.

\subsection{Acknowledgements}
The first-named author thanks the Simons Foundation for partial
support in the form of Collaboration Grant 359573 and Doug Pickrell
for a pointer to the topology literature.  The second-named
author thanks FONDECYT for support from grant 1190066.  Both authors
thank Brian Lawrence for drawing their attention to \cite{ACZpp18} and
an anonymous referee for corrections and for suggesting the
application to bounding heights.

\section{$\EE$ as an analytic surface}\label{s:analytic}
For the rest of the paper, we consider $\EE$ as an analytic surface
(a 2-dimensional complex manifold) and $\CC$ as a Riemann
surface.  Let $\CC^0\subset\CC$ be the open set over which $\pi$ is
smooth and let $\EE^0=\pi^{-1}(\CC^0)$.  Let $j:\CC\to\P^1$ be the
meromorphic function which on $\CC^0$ sends $t$ to the $j$-invariant
of $\pi^{-1}(t)$.

Our goal in this section is to review aspects of the analytic
description of $\EE$ due to Kodaira.  In \cite[\S7]{Kodaira63a},
Kodaira attaches to $\EE$ a period map from the universal cover of
$\CC^0$ to the upper half plane and a monodromy representation from
the fundamental group of $\CC^0$ to $\SL_2(\Z)$.  We assume the reader
is familiar with these invariants.  In \cite[\S8]{Kodaira63a}, Kodaira
reconstructs $\EE$ from this data, and in \cite[\S11]{Kodaira63b}, he
describes the group law on (a subset of) $\EE$ in sheaf theoretic
terms.  We will use these ideas in the rest of the paper to define a
foliation on $\EE$, study its intersections with sections of $\EE$,
and relate them to a certain real-analytic 1-form.

\subsection{Uniformization}\label{ss:unif}
We review the well-known construction of $\EE^0$ as a quotient space.
Let $\widetilde{\CC^0}$ be the universal cover of $\CC^0$, choose a point
$\tilde b\in\widetilde{\CC^0}$, let $b$ be the image of $\tilde b$ in $\CC^0$, and
let $\Gamma=\pi_1(\CC^0,b)$.  Let $\HH$ denote the upper half plane.
Choosing an oriented basis of $H_1(\pi^{-1}(b),\Z)$, we get a period
morphism $\tau:\widetilde{\CC^0}\to\HH$ and a monodromy representation
$\rho:\Gamma\to\SL_2(\Z)$.   We write
\[\rho(\gamma)=\pmat{a_\gamma&b_\gamma\\c_\gamma&d_\gamma}.\]
The period and  monodromy data satisfy the following compatibility:
if $\gamma\in\Gamma$ and $\tilde t\in\widetilde{\CC^0}$, then 
\[\tau(\gamma\tilde t)=\rho(\gamma)\left(\tau(\tilde t)\right)\]
where $\rho(\gamma)$ acts as a linear fractional
transformation on $\HH$.

Form the semi-direct product $\Gamma\ltimes\Z^2$ by using the
monodromy representation and the right
action of $\SL_2(\Z)$ on $\Z^2$:  
\begin{align*}
\left(\gamma_1,m_1,n_1\right)
\left(\gamma_2,m_2,n_2\right)
&=\left(\gamma_1\gamma_2,(m_1,n_1)\gamma_2+(m_2,n_2)\right)\\
&=\left(\gamma_1\gamma_2,
a_{\gamma_2}m_1+c_{\gamma_2}n_1+m_2,b_{\gamma_2}m_1+d_{\gamma_2}n_1+n_2\right).
\end{align*}
For $\tilde t\in\widetilde{\CC^0}$ and $\gamma\in\Gamma$, let 
\[f_\gamma(\tilde t)=\left(c_\gamma\tau(\tilde t)+d_\gamma\right)^{-1}.\]
One checks that $f$ satisfies the cocycle relation
$f_{\gamma_1\gamma_2}(\tilde t)
=f_{\gamma_1}(\gamma_2\tilde t)f_{\gamma_2}(\tilde t)$.

Now let $\Gamma\ltimes\Z^2$ act on $\widetilde{\CC^0}\times\C$ by
\[\left(\gamma,m,n\right)(\tilde t,w)
=\left(\gamma\tilde t,f_{\gamma}(\tilde t)(w+m\tau(\tilde
  t)+n)\right).\]
This action is properly discontinuous, and we have isomorphisms
\[\EE^0\cong\left(\widetilde{\CC^0}\times\C\right)/(\Gamma\ltimes\Z^2)\]
and
\[\CC^0\cong\widetilde{\CC^0}/\Gamma.\]
We will also consider the quotient
\[\FF^0:=\left(\widetilde{\CC^0}\times\C\right)/\Z^2.\]
With these isomorphisms and definition, we may identify the diagram
of complex manifolds
\[\xymatrix{
\left(\widetilde{\CC^0}\times\C\right)/\Z^2\ar[r]\ar[d]
&\left(\widetilde{\CC^0}\times\C\right)/(\Gamma\ltimes\Z^2)\ar[d]\\
\widetilde{\CC^0}\ar[r]&\widetilde{\CC^0}/\Gamma}\]
with the Cartesian diagram
\[\xymatrix{\FF^0\ar[r]\ar[d]&\EE^0\ar[d]\\ \widetilde{\CC^0}\ar[r]&{\CC^0}.}\]

In the introduction, we defined $\omega$ as the line bundle
$O^*(\Omega^1_{\EE/\CC})$.  Let $\omega^{-1}$ be the dual line
bundle.  It is clear from the definitions in this section that a
section of $\omega^{-1}$ over an open set $U\subset\CC^0$ can be
identified with a function $w:\tilde U\to\C$ where $\tilde U$ is the
inverse image of $U$ in $\widetilde{\CC^0}$ and $w$ satisfies
\begin{equation}\label{eq:omega-1}
w(\gamma\tilde t)=f_\gamma(\tilde t)w(\tilde t).
\end{equation}

\subsection{Global monodromy}\label{ss:g-mono}
We recall three well-known results about the monodromy
group $\rho(\Gamma)\subset\SL_2(\Z)$:
\begin{enumerate}
\item $j:\CC\to\P^1$ is non-constant if and only if $\rho(\Gamma)$ is
  infinite, in which case it has
  finite index in $\SL_2(\Z)$.
\item $j:\CC\to\P^1$ is constant if and only if $\rho(\Gamma)$ is
  finite.
\item $\EE$ is constant if and only if $\rho(\Gamma)$ is trivial.
\end{enumerate}

Indeed, if $j$ is non-constant, the period $\tau$ induces a factorization
\[\CC^0\cong\widetilde{\CC^0}/\Gamma\to\HH/\rho(\Gamma)\to
\HH/\PSL_2(\Z)\cong\A^1\]
where the composed map $\CC^0\to\A^1$ is the $j$-invariant.  Since $j$
has finite degree, the index of the image of $\Gamma$ in $\PSL_2(\Z)$
is at most the degree of $j$.

If $j$ is constant, then $\tau$ is constant, and $\FF^0$ is identified
with $\widetilde{\CC^0}\times E_b$ where $E_b:=\pi^{-1}(b)$.  The
action of $\Gamma$ on $\FF^0$ induces an inclusion
$\rho(\Gamma)\subset\aut(E_b)$.  Since the latter has order
$2$, $4$, or $6$, this shows that $\rho(\Gamma)$ is finite.

If $\EE$ is constant, it is clear that $\rho(\Gamma)$ is trivial.
Conversely, if the monodromy is trivial, the argument above shows that
$j$ is constant, $\FF^0\cong\widetilde{\CC^0}\times E_b$, and
$\EE^0\cong\FF^0/\Gamma\cong \CC^0\times E_b$.  Then \cite[p.~585,
$1_1$]{Kodaira63a} shows that the isomorphism
$\EE^0\cong\CC^0\times E_b$ extends to an isomorphism
$\EE\cong\CC\times E_b$.

We say that $\EE\to\CC$ is \emph{isotrivial}
(resp.~\emph{non-isotrivial}) if $j$ is constant (resp.~non-constant).
Obviously, if $\EE$ is constant, it is isotrivial, but not conversely.

\subsection{Local invariants}\label{ss:l-mono}
In this section, we recall from \cite[\S8]{Kodaira63a} the local
monodromy, a branch of the period map, and the line bundle
$\omega^{-1}$ in a neighborhood of each point $t\in\CC$.  We use
Kodaira's notation ($I_0$, $I_0^*$, \dots) for the reduction type of
each fiber to label the rows of the table at the end of the section.

For each $t\in\CC$, let $\Delta_t$ be a neighborhood of $t$
biholomorphic to a disk such that
$\Delta'_t=\Delta_t\setminus\{t\}\subset\CC^0$, and let $z$ be a
coordinate on $\Delta_t$ such that $z=0$ at $t$.

To define the local monodromy, choose a path $p$ from $b$ to a point
of $s\in\Delta'_t$, and let $\gamma$ be a positively oriented loop in
$\Delta'_t$ based at $s$.  Then $\rho$ applied to the class of
$p^{-1}\gamma p$ is an element $g_t\in\rho(\Gamma)\subset\SL_2(\Z)$
which is well defined up to conjugation by $\rho(\Gamma)$.  We say
that $g_t$ is a \emph{generator of the local monodromy at $t$}.  In
the table below, the column ``monodromy'' gives a representative for
the local monodromy for fibers of each type.

If $t\in\CC^0$, the local monodromy is trivial, and the period
map $\tau$ is holomorphic on $\Delta'_t$ and extends to a
holomorphic function on $\Delta_t$.  If $t\in\CC\setminus\CC^0$, the
period map is well-defined on the universal cover
$\widetilde\Delta'_t$ of $\Delta'_t$ and often on a subcover.  In the
table below, the column ``domain'' gives a subcover of
$\widetilde\Delta'_t\to\Delta'_t$ over which the monodromy becomes
trivial, and thus over which a branch of $\tau$ becomes a well-defined
function.  The column ``period'' describes this function for a
suitable choice of a branch of the period map.

We described $\omega^{-1}$ over $\CC^0$ in the last paragraph of
Section~\ref{ss:unif} above.  For $t\in\CC\setminus\CC^0$, we may
specify $\omega^{-1}$ restricted to $\Delta_t$ by giving a section of
$\omega^{-1}$ over $\Delta'_t$ which extends to a generating section
over $\Delta_t$.  Since the monodromy is trivial on the domain, so is
the cocycle $f_\gamma$, and a section of $\omega^{-1}$ on $\Delta'_t$
is a function on the domain.  The column ``generator of
$\omega^{-1}$'' describes this function.

\begin{center}
\renewcommand{\arraystretch}{1.8}
\begin{tabular}{| c | c | c | c | c |}
\hline
Fiber&Monodromy&Domain&Period&Generator of $\omega^{-1}$\\
\hline
$I_0$&$\psmat{1&0\\0&1}$&$z\in\Delta'_t$&$\tau=holo(z)$&$w=1$\\
\hline
$I_b$, $b>0$&$\psmat{1&b\\0&1}$&$e^{2\pi i\zeta}= z$&
$\tau=b\zeta$&$w=1$\\
\hline
$I_b^*$, $b>0$&$\psmat{-1&-b\\0&-1}$&$e^{2\pi i\zeta}= z$&
$\tau=b\zeta$&$w=e^{\pi i\zeta}$\\
\hline
$I_0^*$&$\psmat{-1&0\\0&-1}$&$\zeta^2= z$&
$\tau=holo(z)$&$w=\zeta$\\
\hline
$II$&$\psmat{1&1\\-1&0}$&$\zeta^6=z$&
$\tau=\frac{\eta-\eta^2\zeta^{2h}}{1-\zeta^{2h}}$, $h\equiv1\pmod3$&
$w=\frac{\zeta}{1-\zeta^{2h}}$\\
\hline
$III$&$\psmat{0&1\\-1&0}$&$\zeta^4=z$&
$\tau=\frac{i+i\zeta^{2h}}{1-\zeta^{2h}}$, $h\equiv1\pmod2$&
$w=\frac{\zeta}{1-\zeta^{2h}}$\\
\hline
$IV$&$\psmat{0&1\\-1&-1}$&$\zeta^3=z$&
$\tau=\frac{\eta-\eta^2\zeta^{h}}{1-\zeta^{h}}$, $h\equiv2\pmod3$&
$w=\frac{\zeta}{1-\zeta^{h}}$\\
\hline
$IV^*$&$\psmat{-1&-1\\1&0}$&$\zeta^3=z$&
$\tau=\frac{\eta-\eta^2\zeta^{h}}{1-\zeta^{h}}$, $h\equiv1\pmod3$&
$w=\frac{\zeta^2}{1-\zeta^{h}}$\\
\hline
$III^*$&$\psmat{0&-1\\1&0}$&$\zeta^4=z$&
$\tau=\frac{i+i\zeta^{2h}}{1-\zeta^{2h}}$, $h\equiv1\pmod2$&
$w=\frac{\zeta^3}{1-\zeta^{2h}}$\\
\hline
$II^*$&$\psmat{0&-1\\1&1}$&$\zeta^6=z$&
$\tau=\frac{\eta-\eta^2\zeta^{2h}}{1-\zeta^{2h}}$, $h\equiv2\pmod3$&
$w=\frac{\zeta^5}{1-\zeta^{2h}}$\\
\hline
\end{tabular}
\end{center}
\medskip

In the table, we write $\eta$ for $e^{2\pi i/3}$ and
$holo(z)$ for a holomorphic function on $\Delta'_t$ which extends
holomorphically to $\Delta_t$.

\subsection{Global group structure}\label{ss:g-group}
Let $\EE^{sm}$ be the open subset of $\EE$ where $\pi:\EE\to\CC$ is
smooth, and let $\EE^{id}$ be the union over all $t\in\CC$ of the
identity component of the fiber of $\EE^{sm}$ over $t$.  We may view
$\EE^{id}$ as the sheaf of abelian groups over $\CC$ which assigns to
$U\subset\CC$ the group of holomorphic sections of $\EE^{id}\to\CC$
over $U$.  In \cite[\S11]{Kodaira63b}, Kodaira gives a description of
$\EE^{id}$ in terms of two other sheaves which we now review.

The monodromy representation $\rho$ gives rise to a locally constant
sheaf $\GG_0$ on $\CC^0$ with stalks $\Z^2$.  Taking the direct image
of $\GG_0$ along along the inclusion $\CC^0\subset\CC$ yields
a sheaf $\GG$.  Using the description of the local monodromy in the
preceding section, we see that the stalk of $\GG$ at points of
multiplicative reduction ($I_b$, $b\ge1$) is $\Z$, and the stalk at
points of additive reduction ($I_b^*$, $II$, ...) is $0$.

Using the period morphism $\tau$, we define an inclusion
$\GG\to\omega^{-1}$.  On $\widetilde{\CC}^0$ it sends $\Z^2$ to $\C$
via $(m,n)\mapsto m\tau(\tilde t)+n$, and at points of multiplicative
reduction it sends $\Z\to\C$ via $n\mapsto n$.

Kodaira \cite[Thm~11.2]{Kodaira63b} showed that there is an exact
sequence
\begin{equation}\label{eq:EE-id}
0\to\GG\to\omega^{-1}\to\EE^{id}\to 0
\end{equation}
of sheaves of abelian groups on $\CC$. 

We will use this sequence to work with sections of $\EE$ near bad
fibers.

\section{The Betti foliation}\label{s:Betti}
In this section, we will define a foliation on $\EE^0$ which
has the torsion multisections $\EE^0\cap\EE[n]$ among its closed
leaves.

\subsection{The global Betti foliation}
Given $(r,s)\in\R^2$, consider the set
\[\left\{(\tilde t,r\tau(\tilde t)+s)\left|\ \tilde
      t\in\widetilde{\CC^0}\right.\right\}
\subset \widetilde{\CC^0}\times\C,\]
and define $\FF_{r,s}$ to be its image in $\FF^0$.  Then $\FF_{r,s}$
is a section of the projection $\FF^0\to\widetilde{\CC^0}$ which
depends only on the class of $(r,s)\in(\R/\Z)^2$, and we have an
isomorphism of real analytic manifolds
\[\FF^0=\bigcup_{(r,s)\in(\R/\Z)^2}\FF_{r,s}\cong
  \widetilde{\CC^0}\times(\R/\Z)^2.\] 

We define the \emph{(global) Betti leaf} $\GG_{r,s}$ attached to
$(r,s)\in(\R/\Z)^2$ to be the image of $\FF_{r,s}$ in
$\EE^0\cong\FF^0/\Gamma$.  (This terminology is inspired by
\cite{CorvajaMasserZannier18}, where $r$ and $s$ are called ``Betti
coordinates''. )  The collection of leaves $\GG_{r,s}$ gives a
foliation of $\EE^0$ by immersed analytic submanifolds.  (In other
words, we may give $\GG_{r.s}$ the structure of a complex manifold
such that the inclusion $\GG_{r,s}\to\EE^0$ is an immersion.  The
image is not in general closed, so $\GG_{r,s}$ need not be a
submanifold in the induced topology.)  A straightforward calculation
shows that $\GG_{r,s}=\GG_{r',s'}$ if and only if
\[(r,s)=(r',s')\rho(\gamma)=(a_\gamma r'+c_\gamma s',b_\gamma
  r'+d_\gamma s')\] 
in $(\R/\Z)^2$ for some $\gamma\in\Gamma$.  In particular, the leaves
$\GG_{r,s}$ are in bijection with the orbits of $\Gamma$ acting on
$(\R/\Z)^2$ from the right.

\subsection{The local Betti foliation}\label{ss:l-Betti}
We define local Betti leaves as in \cite{UUpp19}.  Let $V\subset\CC^0$
be non-empty, connected, and simply connected open subset and choose a
lifting $V\to\widetilde{\CC^0}$, $t\mapsto\tilde t$.  Then we get a
branch of the period $\tau:V\to\HH$, $t\mapsto\tau(\tilde t)$, and we
foliate $\pi^{-1}(V)\subset\EE^0$ by leaves $\LL_{r,s}$ where
$\LL_{r,s}$ is the image of the section of $\EE^0\to\CC^0$ given by
\[t\mapsto\text{ the class of }(\tilde t,r\tau(\tilde t)+s)\in
\left(\widetilde{\CC^0}\times\C\right)/(\Gamma\ltimes\Z^2)\cong\EE^0.\]
With this definition we have a trivialization
\[\pi^{-1}(V)\cong V\times(\R/\Z)^2.\]

The following relation between the local and global leaves follows
immediately from the definitions: for $(r,s)\in(\R/\Z)^2$,
\[\pi^{-1}(V)\cap\GG_{r,s}=\bigcup_{(r',s')\in(r,s)\rho(\Gamma)}\LL_{r',s'}.\]  
In other words, over $V$, a global leaf $\GG_{r,s}$ decomposes into
the disjoint union of local leaves, where the union is indexed by the
orbit of the monodromy group on $(\R/\Z)^2$ through $(r,s)$. 

From this we deduce a criterion for a leaf $\GG_{r,s}$ to be closed in
$\EE^0$. 

\begin{prop}\label{prop:closed-leaves}\mbox{}
  \begin{enumerate}
  \item If $\EE$ is isotrivial, every leaf $\GG_{r,s}$ is closed.
  \item If $\EE$ is non-isotrivial, $\GG_{r,s}$ is closed if and only
    if $(r,s)\in(\Q/\Z)^2$ if and only if every point of $\GG_{r,s}$
    is a torsion point in its fiber.
  \item If $\EE$ is not constant, then a section $P$ has image lying
    in a leaf $\GG_{r,s}$ if and only if $P$ is a torsion section.
  \end{enumerate}
\end{prop}

\begin{proof}
  From the local description above, it is clear that $\GG_{r,s}$ is
  closed in $\EE^0$ if the orbit of $\rho(\Gamma)$ through
  $(r,s)$ is finite.  Since $\rho(\Gamma)$ is finite when $\EE$ is
  isotrivial, this establishes part~(1).  

  For part~(2), suppose that $\EE$ is non-isotrivial.  Then as noted
  in Section~\ref{ss:g-mono}, $\rho(\Gamma)$ has finite index in
  $\SL_2(\Z)$.  If $(r,s)\in(\Q/\Z)^2$ it is clear that the orbit
  through $(r,s)$ is finite and that $\GG_{r,s}$ consists of points
  which are torsion in their fiber.  Suppose then that
  $(r,s)\in(\R/\Z)^2\setminus(\Q/\Z)^2$.  It is clear that the points
  of $\GG_{r,s}$ are not torsion in their fiber.  Since $\rho(\Gamma)$
  has finite index in $\SL_2(\Z)$, there is an integer $b$ such that
  $\psmat{1&b\\0&1}\in\rho(\Gamma)$.  If $r\not\in\Q$, then the orbit
  contains 
  \[(r,s)\pmat{1&b\\0&1}^n=(r,nbr+s)\] and thus $\GG_{r,s}$ is not
  closed by Weyl equidistribution.  If $s\not\in\Q$, a similar
  argument shows that $\GG_{r,s}$ is not closed.  This completes the
  proof of part~(2).

  For part~(3), assume that $\EE$ is not constant and that $P$ is a
  section.  If $P$ is torsion, then in every fiber its ``Betti
  coordinates'' $(r,s)$ are rational.  Since $\Q$ is totally
  disconnected, these coordinates must be the same in every fiber, so
  $P$ lies in $\GG_{r,s}$ for some rational pair $(r,s)$.  Conversely,
  if $P$ lies in $\GG_{r,s}$ then $(r,s)$ must be invariant under the
  monodromy group $\rho(\Gamma)$.  Similarly for the multiples $nP$.
  But $\EE$ is non-constant, and this implies that the monodromy group
  is non-trivial and either finite or of finite index in $\SL_2(\Z)$
  (as noted in Section~\ref{ss:g-mono}).  In both cases, it has
  elements with only finitely many fixed points on $(\R/\Z)^2$, so the
  set $\{nP|n\in\Z\}$ is finite, i.e., $P$ is torsion.  This completes the
  proof of part~(3).
\end{proof}

\subsection{Behavior at infinity}
We consider the local geometry of Betti leaves near a singular fiber.
Suppose $t\in\CC\setminus\CC^0$ and, as in Section~\ref{ss:l-mono},
let $\Delta_t\subset\CC$ be a neighborhood of $t$ biholomorphic to a
disk with $\Delta'_t:=\Delta_t\setminus\{t\}\subset\CC^0$.  Let
$V\subset\Delta'_t$ be a non-empty, connected, and simply connected
open set, and define the local monodromy $g_t\in\SL_2(\Z)$ as in
Section~\ref{ss:l-mono} and local Betti leaves $\LL_{r,s}$ as in
Section~\ref{ss:l-Betti}.

We say that a local leaf $\LL_{r,s}$ is \emph{an invariant leaf} (with
respect to $t$) if $(r,s)\in(\R/\Z)^2$ is fixed by
$g_t$ (acting on the right), and we say it is \emph{a vanishing leaf}
(with respect to $t$) if $(r,s)$ is not invariant under
$g_t$.  The latter terminology is motivated by part~(4) of the following result.

\begin{prop}\label{prop:inv-van}\mbox{}
  \begin{enumerate}
  \item If $\LL_{r,s}$ is an invariant leaf, then it extends to a
    section of $\pi:\EE\to\CC$ over $\Delta_t$, and this section meets the
    special fiber $\pi^{-1}(t)$ in a smooth point.
  \item If $\EE$ has multiplicative reduction at $t$ \textup{(}type
    $I_b$, $b\ge1$\textup{)}, let $S\cong(\Z/b\Z)\times S^1$ be the
    closure of the set of points of finite order in the special fiber.
    The invariant leaves extend to sections meeting the special fiber
    at points of $S$, and every point of $S$ is met by the extension of
    a unique invariant leaf $\LL_{r,s}$.
  \item If $\EE$ has additive reduction at $t$ \textup{(}types
    $I_b^*, b\ge0$, $II$, $II^*$, $III$, $III^*$, $IV$, $IV^*$\textup{)}, the
    invariant leaves extend to sections meeting the special fiber at
    one of its finitely many torsion points, 
    and each such point is met by the extension of a unique invariant
    leaf $\LL_{r,s}$.
  \item If $\LL_{r,s}$ is a vanishing leaf, then it extends to a
    connected multisection of $\EE\to\CC$ over $\Delta_t$ of degree
    $>1$ \textup{(}possibly infinite\textup{)}, and this multisection
    meets the special fiber $\pi^{-1}(t)$ in one singular point.
  \end{enumerate}
\end{prop}

\begin{proof}
  Suppose $\LL_{r,s}$ is an invariant leaf.  Then by analytic
  continuation, $\LL_{r,s}$ extends to a section of $\pi$ over
  $\Delta'_t$.  The closure of of this section in $\pi^{-1}(\Delta_t)$
  is proper over $\Delta_t$ (since $\pi$ is proper) and by invariance
  of the intersection number, it meets the special fiber with
  intersection number 1, and thus must meet it at a smooth point.
  This establishes part~(1).

Now assume that $\EE$ has reduction type $I_1$ at $t$.  Let
$\XX=\pi^{-1}(\Delta_t)$ and $\XX'=\pi^{-1}(\Delta'_t)$.  Then Kodaira
showed that 
\[\XX\setminus\XX' = \text{nodal cubic} \cong \C^\times\cup \{q\}\]
where $q$ is the node of the cubic, and that, with a suitable choice
of coordinates,  $\XX'$ has the form
\[\XX'\cong \left(\Delta'_t\times\C^\times\right)/\Z\]
where the action of $\Z$ on $\Delta'_t\times\C^\times$ is
\[m\cdot(u,v)=(u,u^mv).\]
Moreover, there is a holomorphic map
\[\phi:\Delta_t\times\C^\times \to \XX\]
such that $\{t\}\times\C^\times$ maps biholomorphically to the
complement of $q$ in the special fiber, and
$\Delta'\times\C^\times\to\XX'\subset\XX$ is the natural quotient
map.  

In terms of a suitable basis, the local monodromy map is
\[g_t=\pmat{1&1\\0&1}.\]
It is then straightforward to calculate that the invariant leaves are
those of the form $\LL_{0,s}$ for $s\in\R/\Z$.  The corresponding
extended section is 
\[u\mapsto\text{ the class of }(u,e^{2\pi i s}),\]
these sections specialize to points on the unit circle
$S=S^1\subset\C^\times$, and we get the asserted bijection between the
invariant leaves and points on $S$.  The establishes the case $b=1$ of
part~(2).  

The case of $I_b$ reduction for general $b$ is very similar, with
additional notational complexities.  In suitable coordinates, the local
monodromy is
\[g_t=\pmat{1&b\\0&1}\]
and the invariant leaves are those of the form $\LL_{r,s}$ where
$r\in(1/b)\Z/\Z$ and $s\in\R/\Z$.

The smooth part of $\XX=\pi^{-1}(\Delta_t)$ is covered by open subsets
as follows: For $i\in\Z/b\Z$, let
\[W_i=W_i'\cup\C^\times_i,\qquad W_i'=\left(\Delta'_t\times\C^\times\right)/\Z\]
where the action of $\Z$ on $\Delta_t'\times\C^\times$ is
\[m\cdot(u,v)=(u,u^{bm}v).\]  For
$u\in\Delta'_t$ and $v\in\C^\times$, write $(u,v)_i$ for the
class of $(u,v)$ in $W_i'$.  Then $\XX^{sm}$ is obtained by
gluing the  $W_i$ according to the rule
\[(u,v)_i=(u,u^{j-i}v)_j\]
for all $u\in\Delta'_t$, $v\in\C^\times$, and $i,j\in\Z/b\Z$. 

The invariant leaf $\LL_{i/b,s}$ lies in the open
corresponding to $i$ and extends to the section 
\[u\mapsto\text{ the class of }(u,e^{2\pi i s})_i,\] 
and we find that the specializations of extensions of invariant leaves
are in bijection with 
\[(1/b)\Z/\Z\times S^1\subset \pi^{-1}(t),\] 
as required.  This establishes part~(2) in the general case.

For part~(3), recall the explicit generators for the local monodromy
groups in the table at the end of Section~\ref{ss:l-mono} .  Using
these, one computes the invariant leaves, which are as follows:
\begin{align*}
I_b^*, b\text{ odd}:&\LL_{0,0}, \LL_{1/2,1/4},\LL_{0,1/2},\LL_{1/2,3/4}\\
I_b^*, b\text{ even}:&\LL_{0,0}, \LL_{1/2,1/2},\LL_{0,1/2},\LL_{1/2,0}\\
II, II^*:&\LL_{0,0}\\
III, III^*:&\LL_{0,0},\LL_{1/2,1/2}\\
IV, IV^*:&\LL_{0,0},\LL_{1/3,2/3},\LL_{2/3,1/3}.
\end{align*}
The results of Kodaira recalled in Section~\ref{ss:g-group} show that
for these reduction types, the connected component of the special
fiber is isomorphic to the additive group (and so is torsion free),
and the group of torsion points on the special fiber is isomorphic to
the group of components.  
It is then straightforward to see that each of the corresponding
sections specializes to a torsion point and that all torsion points
on the special fiber are met by the extension of a unique invariant
leaf. 

For part~(4), it is clear that analytic continuation of a vanishing
leaf $\LL_{r,s}$ yields a multisection over $\Delta'_t$ whose degree
is the order of the orbit of the monodromy group through $(r,s)$,
which by assumption is $>1$.  That its closure in $\pi^{-1}(\Delta_t)$
adds a single point over $t$ which is singular in the special fiber
requires a tedious analysis of cases.  Since we will not use this
result elsewhere in the paper, we omit the details.
\end{proof}

\section{Intersections with the Betti foliation}\label{s:intersections}
For the rest of the paper, we assume that $\pi:\EE\to\CC$ is
non-constant and that $P$ is not torsion.  In this section,
we will quantify tangencies between $P$ and the Betti foliation in
terms of intersection numbers.

\subsection{Local intersection numbers}\label{ss:local-int}
Suppose first that $t\in\CC^0$, i.e., that $\EE$ has good reduction at
$t$.  Over a neighborhood of $t$, there is a unique local Betti leaf
$\LL$ passing through $P(t)$.  Since $P$ is not torsion,
Proposition~\ref{prop:closed-leaves}(3) implies that this intersection
is isolated, i.e., by shrinking the neighborhood, we may assume $P$ and
$\LL$ meet only over $t$.  We define $I(P,t)$ to be the intersection
multiplicity of $P$ and $\LL$ at $P(t)$.  (This is the local
intersection number of two holomorphic 1-manifolds meeting at an isolated
point of a holomorphic 2-manifold.  We will make it explicit in terms
of the order of vanishing of a holomorphic function below.)

Note that the intersection in question satisfies $I(P,t)\ge1$, and
$I(P,t)\ge2$ if and only if $P$ is tangent to $\LL$ at $t$, i.e.,
if and only if $t\in T_{Betti}$.

Now assume that $t\in\CC\setminus\CC^0$.  Let $S\subset\pi^{-1}(t)$ be
the closure of the set of torsion points in the special fiber.
As noted in Proposition~\ref{prop:inv-van}, $S\cong(\Z/b\Z)\times S^1$
if $\EE$ has reduction type $I_b$ at $t$, and it is a finite group in
the other cases.  If $P(t)\not\in S$, we define $I(P,t)=0$.  If
$P(t)\in S$, then by Proposition~\ref{prop:inv-van}, there is a unique
invariant local leaf $\LL$ extending over a neighborhood of $t$ and
meeting $P$ over $t$.  We define $I(P,t)$ to be the intersection number
of $P$ and this extended leaf at $t$.

\begin{lemma}\label{lemma:nP}
  For all integers $n>0$ and all points $t\in\CC$, 
\[I(P,t)=I(nP,t).\]
\end{lemma}

\begin{proof}
  Indeed, since the multiplication by $n$ map $\EE^{sm}\to\EE^{sm}$ is
  \'etale, for every $(r,s)\in(\R/\Z)^2$, the intersection number of
  $P$ with the local leaf $\LL_{r,s}$ at $t$ is the same as the
  intersection number of $nP$ with $\LL_{nr,ns}$ at $t$.
\end{proof}

\subsection{Explicit intersection numbers}\label{ss:explicit}
In this section, we make the intersection number $I(P,t)$ more explicit
by using the exact sequence \eqref{eq:EE-id}.  

By Lemma~\ref{lemma:nP}, we may replace $P$ with a multiple and
thereby assume that $P$ passes through the identity component of each
fiber.

Fix $t\in\CC$ and choose a small enough neighborhood $\Delta_t$ of $t$
in $\CC$ such that the restricted section $P:\Delta_t\to\EE$ lifts to
a section of $\omega^{-1}$ over $\Delta_t$ and such that $\omega^{-1}$
is trivial over $\Delta_t$.  Let $z$ be a coordinate on $\Delta_t$
such that $t$ corresponds to $z=0$.  We may then identify $P$ with a product
$w=hw_0$ where $h$ is a holomorphic function on $\Delta_t$ and $w_0$ is
a generating section of $\omega^{-1}$ over $\Delta_t$ as specified in
the table at the end of Section~\ref{ss:l-mono}.

The local multiplicity $I(P,t)$ is by definition the intersection
number of $P$ and an invariant local Betti leaf $\LL_{r,s}$.  Since
the leaf is invariant, the map $z\mapsto r\tau(z)+s$ defines a section of
$\omega^{-1}$ over $\Delta_t$, and the intersection multiplicity is the
same as the intersection number between the graphs of the functions
$z\mapsto h(z)$ and $z\mapsto (r\tau(z)+s)/w_0$

If $t\in\CC^0$, then $w_0=1$.   If $h(t)=r\tau(t)+s$,  $I(P,t)$ is the
intersection number between the graph of $z\mapsto h(z)$ and the graph
of $z\mapsto r\tau(z)+s$.  Therefore,
\begin{equation}\label{eq:I-good}
I(P,t)=\ord_{z=0}\left(h(z)-r\tau(z)-s\right).  
\end{equation}

If $\EE$ has multiplicative reduction ($I_b$, $b>0$) at $t$, then
$w_0=1$ and $I(P,t)=0$ if $h(t)\not\in\R$.  If $h(t)=s\in\R$, then
$I(P,t)$ is the intersection number between the graph of
$z\mapsto h(z)$ and the graph of the constant function $z\mapsto s$.
Therefore,
\begin{equation}\label{eq:I-mult}
I(P,t)=
\begin{cases}
0&\text{if $h(t)\not\in\R$}\\
\ord_{z=0}\left(h(z)-s\right)&\text{if $h(t)=s\in\R$.}
\end{cases}  
\end{equation}

If $\EE$ has additive reduction at $t$ (types $I_b^*$, $II$, ...),
then $I(P,t)=0$ unless $h(t)=0$, and if $h(t)=0$, then $I(P,t)$ is the
intersection number between the graph of $z\mapsto h(z)$ and the graph
of $z\mapsto 0$.  Therefore,
\begin{equation}\label{eq:I-add}
I(P,t)=\ord_{z=0}\left(h(z)\right).  
\end{equation}

\section{A real analytic 1-form}\label{s:1-form}
In this section, we review a connection between local and global
degrees of smooth sections of a line bundle.  We then construct a real
analytic 1-form whose zeroes will turn out to control tangencies
between a section $P$ and the Betti foliation.

\subsection{Local and global indices}
The number of zeroes and poles of a meromorphic section of a line
bundle (counted with multiplicities) is the degree of the line bundle.
This familiar result from basic algebraic geometry is in fact purely
topological.  In this section, we state and sketch the proof of the
result in the smooth category.  Our
Proposition~\ref{prop:zeroes-degree} is in substance equivalent to
\cite[Thm.~11.17]{BottTuDFIAT}, but the language there is rather
different than ours, so for the convenience of the reader, we review
the main lines of the argument adapted to our situation.

\subsubsection{Winding numbers}\label{ss:winding}
Let $\Delta$ be the unit disk in $\C$, and let
$\Delta'=\Delta\setminus\{0\}$.  Suppose that $f$ is a smooth, nowhere
vanishing, complex-valued function on $\Delta'$.  We define the
\emph{winding number} of $f$ to be
\[W(f):=\frac1{2\pi i}\oint d\log f\]
where the path of integration is any positively oriented loop around 0.
Equivalently
\[W(f)=\frac1{2\pi i}\int_0^1\frac{g'(t)}{g(t)}dt\]
where $g(t)=f(re^{2\pi it})$ for some $0<r<1$.

The following properties of $W(f)$ are well known.  See, for example,
\cite[Ch.~3]{FultonAT}. 
\begin{enumerate}
\item $W(f)$ is an integer and is independent of the choice of path of
  integration.
\item If $f$ extends to a smooth nowhere vanishing function
  on $\Delta$, then $W(f)=0$.
\item $W(f_1f_2)=W(f_1)+W(f_2)$.
\item If $f$ is the restriction of a meromorphic function on $\Delta$,
  then $W(f)=\ord_{z=0}f(z)$.
\item If $F(\sigma,z)$ is a smooth, nowhere vanishing function on
  $[0,1]\times\Delta'$ and $f_\sigma(z)=F(\sigma,z)$, then $W(f_0)=W(f_1)$.
\end{enumerate}

The following is essentially the ``dog on a leash'' theorem, see
\cite[Thm~3.11]{FultonAT}. 
\begin{lemmass}\label{lemma:W-sum}\mbox{}
  \begin{enumerate}
  \item  Suppose that $f_1$ and $f_2$ are smooth functions
  on the punctured disk $\Delta'$, and let $f=f_1-f_2$.
  Suppose also that there exist real numbers $m_1<m_2$
  and positive real numbers $C_1$, and $C_2$ such that
\begin{align*}
  |f_1(z)|&\ge C_1|z|^{m_1}\\
\noalign{and}
  |f_2(z)|&\le C_2|z|^{m_2}
\end{align*}
for all $\in\Delta'$.  Then $W(f)=W(f_1)$.  
\item The same conclusion holds when $m_1=m_2$ provided that $C_1>C_2$.
  \end{enumerate}
 \end{lemmass}

\begin{proof}
  Define $F$ on $[0,1]\times\Delta'$ by
  $F(\sigma,z)=f_1(z)-\sigma f_2(z)$, so that $F(1,z)=f(z)$ and
  $F(0,z)=f_1(z)$.  The displayed inequalities show that
  $F(\sigma,z)\neq0$ for all sufficiently small $z$, so we may shrink
  $\Delta'$ and have that $F(\sigma,z)$ is nowhere vanishing on
  $[0,1]\times\Delta'$.  The winding numbers $W(f)$ and $W(f_1)$ are
  then well defined, and property (5) of winding numbers shows that
  $W(f)=W(f_1)$.
\end{proof}

\subsubsection{Local indices}
Let $L$ be a holomorphic line bundle on $\CC$ and suppose that $s$ is
a smooth, nowhere vanishing section of $L$ over an open subset of the
form $U=\CC\setminus\{t_1,\dots,t_m\}$.  We define a local index
$J(s,t)$ for all $t\in\CC$ as follows: Given $t$, choose a
neighborhood $U_t$ of $t$ diffeomorphic to a disk and such that $s$ is
defined and non-zero on $U'_t:=U_t\setminus \{t\}$.  Choose a
trivializing section $s_t$ of $L$ (as a complex line bundle) over
$U_t$, and write $s=f(z)s_t$ for $z\in U'_t$.  Then
\[J(s,t):=W(f)\] 
where we identify $f$ with a function on $\Delta'$ via a
diffeomorphism $\Delta\cong U_t$ sending $0$ to $t$.  The properties
of $W$ recalled above imply that $J(s,t)$ is an integer and is
independent of the various choices.  The also imply that if $s$ is a
meromorphic section of $L$ near $t$, then $J(s,t)$ is exactly the
order of zero or pole of $s$ at $t$ in the usual sense.

The following global result generalizes the statement that the sum of
the orders of zero or pole of a meromorphic section of a line bundle is
the degree of the line bundle.  Recall that $H^2(\CC,\Z)$ is
canonically isomorphic to $\Z$.  We define $\deg(L)$ to be the
first Chern class $c_1(L)\in H^2(\CC,\Z)=\Z$.

\begin{propss}\label{prop:zeroes-degree} 
  Suppose $s$ is a smooth, nowhere vanishing section of $L$ over
  $U=\CC\setminus\{t_1,\dots,t_m\}$.  Then
\[\sum_{t\in\CC}J(s,t)=\sum_{i=1}^mJ(s,t_i)=\deg(L).\]
\end{propss}

\begin{proof}
  Property~(2) of winding numbers recalled above implies that
  $J(s,t)=0$ unless $t$ is in $\{t_1,\dots,t_m\}$, so the sum over
  $t\in\CC$ is well defined and equal to the sum over the $t_i$.  To
  prove the equality with the degree of $L$ we will compare \v{C}ech and
  de~Rham cohomologies.

  For $i=1,\dots,m$, let $U_i$ be a neighborhood of $t_i$
  diffeomorphic to the disk $\Delta$ with $0$ corresponding to $t_i$
  and such that the closures of the $U_i$ in $\CC$ are disjoint.
  Choose simply connected open sets $U_{m+1},\dots,U_n\subset U$ such
  that $U_1,\dots,U_n$ covers $\CC$ and such that
  $U_{ij}:=U_i\cap U_j$ is simply connected for all pairs of indices
  $1\le i,j\le n$.  For $i=1,\dots,m$, choose generating sections
  $s_i$ of $L$ over $U_i$, and for $i=m+1,\dots,n$, let $s_i$ be the
  restriction of $s$ to $U_i$.  Then there are smooth, nowhere
  vanishing functions $g_{ij}$ defined on $U_{ij}$ by
\[s_i=g_{ij}s_j,\]
and the $g_{ij}$ form a 1-cocycle with values in $\AA^\times$, the
sheaf of nowhere vanishing smooth functions on $\CC$.  The class of
this cocycle in \v{C}ech cohomology is $[L]\in H^1(\CC,\AA^\times)$.  

We have an exact sequence
\[0\to\Z\to\AA\to\AA^\times\to0\] 
where $\AA$ is the sheaf of smooth functions on $\CC$ and the map
$\AA\to\AA^\times$ is $f\mapsto e^{2\pi i f}$.  Taking the coboundary
of $[L]$ in the long exact sequence of cohomology, we find that
$c_1(L)\in H^2(\CC,\Z)$ is represented by the 2-cocycle
\[\eta_{ijk}=\frac1{2\pi i}\left(\log g_{ij}-\log g_{ik}+\log g_{jk}\right).\]

Next, we write down a 2-form representing the image of $[L]$ under
\[H^2(\CC,\Z)\to H^2(\CC,\C)\cong H^2_{dR}(\CC)\tensor\C.\]
Since $\eta_{ijk}$ is $\Z$-valued, we have $d\eta_{ijk}=0$.  This
implies that
\[h_{ij}=\frac1{2\pi i}d\log g_{ij}\]
is a 1-cocycle with values in $\AA^1$, the sheaf of smooth 1-forms on $\CC$.

Now choose a partition of unity $\rho_i$ subordinate to the cover
$U_i$ of $\CC$.  Shrinking $U_i$ for $i=m+1,\dots,n$ if necessary, we
may assume that for $i=1,\dots,m$, there are closed disks of positive
radius $K_{i,1}\subset K_{i,2}\subset U_i$ such that $\rho_i$ is
identically 1 on $K_{i,1}$ and identically zero on the complement of
$K_{i,2}$.

Setting
\[\theta_i=\frac1{2\pi i}\sum_{\ell=1}^n
\rho_\ell \,d\log g_{i\ell}\in\AA^1(U_{i})\]
we see that $\theta_i-\theta_j=h_{ij}$.  Since $h_{ij}$ is $d$-closed
for all $ij$, we find that $d\theta_i=d\theta_j$ on $U_{ij}$ and so
we may define a global 2-form $\Omega$ on $\CC$ by requiring that
\begin{align*}
\Omega&=-d\theta_i  \\
&=\frac{-1\phantom{-}}{2\pi i}\sum_{\ell=1}^nd\rho_\ell\,d\log g_{i\ell}
\end{align*}
on $U_i$.  

It follows from the ``generalized Mayer-Vietoris principle''
\cite[\S8]{BottTuDFIAT} (also known fondly to some as the ``\v{C}ech-de~Rham
shuffle''), that $\Omega$ represents the class of $L$ in de~Rham
cohomology.   More formally
\[
c_1(L)=\int_\CC\Omega.
\]
(The point is that the $\Z$-valued 2-cocycle $\eta_{ijk}$ and the
$\AA^2$-valued 0-cocycle $\Omega$ represent the same class in the
cohomology of the total complex of the \v{C}ech-deRham double complex
because, by construction, they differ by a coboundary.)

To finish the proof, we will relate the displayed integral to winding numbers.
Let $U_i'=U_i\setminus\{t_i\}$, and let $g_i\in\AA^\times(U_i')$ be
defined by
\[s_i=g_is.\]
Then examining the definitions shows that
\begin{align*}
  g_{ij}=1&\quad\text{if $i,j>m$}\\
g_{ij}=g_i&\quad\text{if $i\le m$ and $j>m$}\\
U_i\cap U_j=\emptyset&\quad\text{if $i,j\le m$.}
\end{align*}
It follows that $\Omega$ vanishes identically on the complement of
$\cup_{i=1}^m U_i$, and so
\begin{equation}\label{eq:deg-sum}
\deg(L)=\sum_{i=1}^m\int_{U_i}\Omega.  
\end{equation}

On $U_i$ we have
\begin{align*}
  \Omega_{|U_i}
&=\frac{-1\phantom{-}}{2\pi i}
\sum_{\ell=m+1}^nd\rho_\ell\,d\log g_i\\
&=\frac{1}{2\pi i}
d\rho_i\,d\log g_i.
\end{align*}
Now $d\rho_i$ is identically zero on $K_{i,1}$ and on the complement
of $K_{i,2}$, so we have
\begin{align*}
  \int_{U_i}\Omega
&=\frac{1}{2\pi i}\int_{K_{i,2}\setminus K_{i,1}}d\rho_i\,d\log g_i\\
&=\frac{1}{2\pi i}\int_{\partial K_{i,2}}\rho_i\,d\log g_i
-\frac{1}{2\pi i}\int_{\partial K_{i,1}}\rho_i\,d\log g_i
\end{align*}
by Stokes' theorem.  Since $\rho_i$ vanishes on $\partial K_{i,2}$ and
is 1 on $\partial K_{i,1}$, we find
\[\int_{U_i}\Omega=-W(g_i)=W(g_i^{-1})=J(s,t_i)\]
where the last equality follows from the definition of $J$.
Combining this with Equation~\eqref{eq:deg-sum}, we find that
\[\deg(L)=\sum_{i=1}^mJ(s,t_i)\]
as desired.  This completes the proof of the proposition.
\end{proof}

\subsection{Constructing $\eta$}
We use the uniformization of Section~\ref{ss:unif}.  Let $w$ be the
standard coordinate on $\C$ and recall the period function
$\tau:\widetilde{\CC^0}\to\HH$.  Consider the real-analytic 1-form on
$\widetilde{\CC^0}\times\C$ given by
\[\tilde\eta=dw-\frac{\im w}{\im \tau}d\tau.\]
Under the action of $\Gamma\ltimes\Z^2$, straightforward calculation
shows that 
\[(id,m,n)^*(\tilde\eta)
=\tilde\eta\]
and
\[(\gamma,0,0)^*(\tilde\eta)=f_\gamma\tilde\eta
=\frac{\tilde\eta}{c_\gamma\tau+d_\gamma}
\]
where as usual
$\rho(\gamma)=\psmat{a_\gamma&b_\gamma\\c_\gamma&d_\gamma}$.
These formulas show that $\tilde\eta$ descends to a real-analytic
section $\eta$ of
\[\Omega^1_{\EE^0}\tensor\pi^*(\omega)^{-1}
=\Omega^1_{\EE^0}\tensor\left(\Omega^1_{\EE^0/\CC^0}\right)^{-1}.\]

It is immediate from the definition of the Betti foliation in terms of
the uniformization $\widetilde{\CC^0}\times\C\to\EE^0$ that at every
point $x\in\EE^0$, the kernel of $\eta$ as a functional on the
holomorphic tangent space of $\EE^0$ at $x$ is precisely the tangent space to
the leaf of the Betti foliation passing through $x$.  We will thus be
able to use $\eta$ to quantify the tangencies between sections $P$ and
the Betti foliation.


\subsection{Definition of $\eta_P$}\label{ss:etaP-def}
Now assume that $P$ is a non-torsion section of $\EE\to\CC$ and recall
that the latter is assumed to be non-constant.  Let
$\eta_P:=P^*(\eta)$.  This is a real analytic section of
$\Omega^1_\CC\tensor\omega^{-1}$ over $\CC^0$.  Since the kernel of
$\eta$ at a point of $\EE^0$ is the tangent space to the leaf of the
Betti foliation through that point, we see that $\eta_P$ vanishes at a
point of $\CC^0$ if and only if that point lies in
$T_{Betti}\cap\CC^0$.  Since $T_{Betti}$ is finite by
\cite[\S3]{UUpp19}, it follows that $\eta_P$ has only finitely many
zeroes.  Thus, Proposition~\ref{prop:zeroes-degree} applies, and we
have the following key result.

\begin{prop}\label{prop:EtaP-degree}
\[\sum_{t\in\CC}J(\eta_P,t)=
\deg(\Omega^1_\CC\tensor\omega^{-1})=
2g-2-d.\]
\end{prop}

We end this section with a lemma parallel to Lemma~\ref{lemma:nP}.

\begin{lemma}\label{lemma:nP'}
  For all integers $n>0$ and all points $t\in\CC$, 
\[J(\eta_P,t)
=J(\eta_{nP},t).\]
\end{lemma}

\begin{proof}
It is clear from the local expression for $\eta$ as 
\[dw-\frac{\im w}{\im\tau}d\tau\]
the $\eta_{nP}=n\eta_P$.  The equality of local indices then
follows from properties (3) and (4) of the winding number $W$.
\end{proof}

\section{Zeroes and intersection numbers}\label{s:zeros}
In this section, we relate the intersection number $I(P,t)$ to the local index
$J(\eta_P,t)$.  

\begin{prop} \label{prop:int-van}
For all $t\in\CC$
  \[J(\eta_P,t)=I(P,t)-1.\]
\end{prop}

\begin{proof}
  By Lemmas~\ref{lemma:nP} and \ref{lemma:nP'}, we may replace $P$
  with a multiple and reduce to the case where $P$ passes through the
  identity component of every fiber of $\EE\to\CC$.

  Fix $t\in\CC$ and let $\Delta_t$ be a neighborhood of $t$
  biholomorphic to a disk and such that
  $\Delta'_t:=\Delta_t\setminus\{t\}$ lies in $\CC^0$.  Let $z$ be a
  coordinate on $\Delta_t$ such that $t$ corresponds to $z=0$.  Recall
  from Section~\ref{ss:explicit} that shrinking $\Delta_t$ if
  necessary, we may lift $P$ to $\omega^{-1}$ in the exact sequence of
  Equation~\eqref{eq:EE-id} and identify the lift with a product
  $w=hw_0$ where $w_0$ is a generating section of $\omega^{-1}$ (as
  specified in the table at the end of Section~\ref{ss:l-mono}) and
  $h$ is a holomorphic function on $\Delta_t$.  In terms of this data,
  we have
\[\eta_P=d(hw_0)-\frac{\im hw_0}{\im \tau}d\tau.\]
The winding number that defines $J(\eta_P,t)$ is then $W(f)$ where
\[f=\frac1{w_0}\left(\frac{d(hw_0)}{dz}
-\frac{\im (hw_0)}{\im \tau}\frac{d\tau}{dz}\right).\]

To lighten notation, let $n=I(P,t)$, so that our goal is to prove that
$W(f)=n-1$.  We will complete the proof of the proposition in the next
four sections, dividing into cases according to the reduction of $\EE$
at $t$.

\subsection{Points of good reduction}
If $\EE$ has good reduction at $t$, then we saw in
Equation~\eqref{eq:I-good} that
\[n:=I(P,t)=\ord_{z=0}\left(h(z)-r\tau(z)-s\right),\]
where $h(t)=r\tau(t)+s$.
Since $w_0=1$, we have $J(\eta_P,t)=W(f)$ where
\[f=\frac{dh}{dz}-\frac{\im h}{\im \tau}\frac{d\tau}{dz}.\]

Let  
\[f_1(z)=\frac{dh}{dz}-r\frac{d\tau}{dz}
\quad\text{and}\quad
f_2(z)=\left(\frac{\im h}{\im \tau}-r\right)\frac{d\tau}{dz}.\]

Since $h-r\tau-s$ is holomorphic and vanishes to order $n\ge1$ at $t$,
we have
\[|f_1(z)|\ge C_1|z|^{n-1}\]
for some positive constant $C_1$ and all sufficiently small $z$.  On
the other hand, 
\[\im h-r\im \tau=\frac12\left((h-r\tau-s)-(\overline{h-r\tau-s})\right),\]
$\im\tau(t)>0$, and $\tau$ is holomorphic on $\Delta$, so 
\[|f_2(z)|\le C_2|z|^n\]
for some positive constant $C_2$ and all sufficiently small $z$.
Applying Lemma~\ref{lemma:W-sum}, we have $W(f)=W(f_1)$, and since
$f_1$ is holomorphic on $\Delta$ and vanishes to order $n-1$ at $z=0$,
we have $W(f)=W(f_1)=n-1$.  This establishes that
$J(\eta_P,t)=I(P,t)-1$ for all $t\in\CC^0$.

\subsection{Points of multiplicative reduction}
Next assume that $\EE$ has reduction type $I_b$ ($b\ge1$) at $t$.
According to Equation~\eqref{eq:I-mult}, 
\[n:=I(P,t)=
\begin{cases}
0&\text{if $h(t)\not\in\R$}\\
\ord_{z=0}\left(h(z)-s\right)&\text{if $h(t)=s\in\R$}
\end{cases}
\]
and by the first part of the proof, $J(\eta_P,t)=W(f)$ where
\[f=\frac1{w_0}\left(\frac{d(hw_0)}{dz}
-\frac{\im (hw_0)}{\im \tau}\frac{d\tau}{dz}\right).\]

According to the table at the end of Section~\ref{ss:l-mono}, $w_0=1$
and  $\tau=(b/2\pi i)\log z$, so
\[f=\frac{dh}{dz}-\frac{i\im h}{z\log|z|}.\]

Suppose that $h(t)\not\in\R$.  Then after shrinking $\Delta_t$, we
may assume that $|h(z)|$ and $|(\im h(z))|$ are bounded above and below on
$\Delta_t$ by positive constants.   Since $dh/dz$ is holomorphic, we
have that $|dh/dz|$ is bounded above on $\Delta_t$ as well.  On the other
hand $|1/(z\log|z|)|>C|z^{-1+\epsilon}|$ for all $\epsilon>0$.  Thus,
setting $f_1=(i\im h)/(z\log|z|)$ and $f_2=dh/dz$,
Lemma~\ref{lemma:W-sum} implies that $W(f)=W(f_1)$.
Finally,
\[W\left(\frac{i\im h}{z\log|z|}\right)
=W(i\im h)-W(z)-W(\log|z|)=0-1-0=-1.\]
Thus we find that $W(f)=-1=n-1$,
establishing that $J(\eta_P,t)=I(P,t)-1$ when $h(t)\not\in\R$.

To finish the multiplicative case, assume that $h(t)=s\in\R$.  Then
$|dh/dz(z)|\ge C_1|z|^{n-1}$ for some positive $C_1$ where
$n=\ord_{z=0}(h(z)-s)$. Also,  $\im h\le C_2|z|^n$, and we find that 
$|(i\im h)/(z\log|z|)|\le C_2|z|^{n-1}/|\log|z||$.   Shrinking
$\Delta'$, we may assume that $C_2<C_1$ and $|(i\im h)/(z\log|z|)|\le
C_2|z|^{n-1}$.  Setting $f_1=dh/dz$ and $f_2=(i\im h)/(z\log|z|)$,
Lemma~\ref{lemma:W-sum} implies that $W(f)=W(f_1)$.  Since $f_1$ is
holomorphic with $\ord_{z=0}(f_1)=n-1$, we conclude that $W(f)=n-1$.
This establishes that $J(\eta_P,t)=I(P,t)-1$ when $h(t)\in\R$ and
completes the proof for places $t$ of multiplicative reduction.

\subsection{Points of potentially multiplicative reduction}
Now assume that $\EE$ has reduction type $I_b^*$ ($b>0$) at $t$.  
According to Equation~\eqref{eq:I-add}, 
\[n:=I(P,t)=\ord_{z=0}\left(h(z)\right)\]
and by the first part of the proof, $J(\eta_P,t)=W(f)$ where
\[f=\frac1{w_0}\left(\frac{d(hw_0)}{dz}
-\frac{\im (hw_0)}{\im \tau}\frac{d\tau}{dz}\right).\]

According to the table at the end of Section~\ref{ss:l-mono}, $w_0=z^{1/2}$
and  $\tau=(b/2\pi i)\log z$, so
\[f=\frac{dh}{dz}+\frac12\frac hz
-\frac{\im(hz^{1/2})}{z^{1/2}}\frac{i}{z\log|z|}.\]

Note that 
\[\left|\frac{dh}{dz}+\frac12\frac hz\right|\ge
C_1|z|^{n-1}\]
on $\Delta'_t$ for some positive constant $C_1$.
On the other hand,  $|\im(hz^{1/2})/z^{1/2}|\le C|z|^n$ on
$\Delta'_t$ for some positive constant $C$.  Thus
\[\left|\frac{\im(hz^{1/2})}{z^{1/2}}\frac{i}{z\log|z|}\right|\le
C\frac{|z|^{n-1}}{|log|z||}.\]
Shrinking $\Delta$, we may ensure that 
\[\left|\frac{\im(hz^{1/2})}{z^{1/2}}\frac{i}{z\log|z|}\right|\le
C_2|z|^{n-1}\]
for some positive $C_2<C_1$.  Setting 
\[f_1=\frac{dh}{dz}+\frac12\frac hz \quad\text{and}\quad
  f_2=\frac{\im(hz^{1/2})}{z^{1/2}}\frac{i}{z\log|z|},\]
Lemma~\ref{lemma:W-sum} implies that $W(f)=W(f_1)$, and since $f_1$ is
holomorphic on $\Delta_t$ and vanishes to order $n-1$ at $z=0$, we
have $W(f_1)=n-1$.  This establishes that $J(\eta_P,t)=I(P,t)-1$ when
$\EE$ has reduction type $I_b^*$ at $t$, completing the proof for
places $t$ of potentially multiplicative reduction.

\subsection{Points of potentially good reduction}
Now assume that the reduction type of $\EE$ at $t$ is one of those of
additive, potentially good reduction, namely $I_0^*$, $II$, $III$,
$IV$, $IV^*$ $III^*$, or $II^*$.  According to
Equation~\eqref{eq:I-add},
\[n:=I(P,t)=\ord_{z=0}\left(h(z)\right)\]
and by the first part of the proof, $J(\eta_P,t)=W(f)$ where
\[f=\frac1{w_0}\left(\frac{d(hw_0)}{dz}
-\frac{\im (hw_0)}{\im \tau}\frac{d\tau}{dz}\right).\]

Using the table at the end of Section~\ref{ss:l-mono}, we see that
$w_0$ is a fractional power of $z$ times a non-vanishing, holomorphic
function of $\zeta$ on the domain listed in the table.  From this we
calculate that
\[f_1:=\frac1{w_0}\left(\frac{d(hw_0)}{dz}\right)
=\frac{dh}{dz}+\alpha\frac{h}{z}+g\]
where $\alpha\in\{1/6,1/4, 1/3,1/2,2/3,3/4,5/6\}$ and $g$ is
holomorphic and vanishes to order $\ge n$.  Thus $f_1$ is holomorphic
on $\Delta_t$ and
\[|f_1(z)|\ge C_1|z|^{n-1}\]
for some positive constant $C_1$.  

Now consider 
\[f_2=\frac{\im (hw_0)}{w_0\im \tau}\frac{d\tau}{dz}.\]
Since $\ord_{z=0}(h(z))=n$, and $\im \tau$ is bounded away from zero,
we find that 
\[\left|\frac{\im (hw_0)}{w_0\im \tau}\right|\le C|z|^n\]
on $\Delta'_t$ for some positive constant $C$.
On the other hand, since $\tau$ is a holomorphic function of $\zeta$,
and $z=\zeta^b$ with $b\in\{1,3,4,6\}$, we see that 
\[\left|\frac{d\tau}{dz}(z)\right|\le C'|z|^{-\beta}\]
with $\beta\in\{0,2/3,3/4,5/6\}$ for some positive $C'$.  It follows that
\[|f_2(z)|\le C_2|z|^{n-\beta}\]
with $n-\beta>n-1$.  Applying Lemma~\ref{lemma:W-sum} we find that
$W(f)=W(f_1)=n-1$. 
This establishes that $J(\eta_P,t)=I(P,t)-1$ when
$\EE$ has additive and potentially good reduction type at $t$.

This completes the proof of Proposition~\ref{prop:int-van} in all cases.
\end{proof}

\section{Proof of main theorems}\label{s:proofs}
The key result of this paper is the following equality.

\begin{thm}\label{thm:total-int}
  Suppose that $\pi:\EE\to\CC$ is non-constant and that $P$ is a
  section of $\pi$ of infinite order.  Let $g$ be the genus of $\CC$
  and let $d$ be the degree of the line bundle
  $\omega=O^*(\Omega^1_{\EE/\CC})$.  Let $I(P,t)$ be the local
  intersection indices defined in Section~\ref{ss:local-int}.  Then
\[\sum_{t\in\CC}\left(I(P,t)-1\right)=2g-2-d.\]
\end{thm}

\begin{proof}
  Let $\eta_P$ be the 1-form attached to $P$ in
  Section~\ref{ss:etaP-def}.  Then according to
  Proposition~\ref{prop:EtaP-degree}, we have
\[2g-2-d=\sum_{t\in\CC}J(\eta_P,t)\]
and according to Proposition~\ref{prop:int-van}
  \[J(\eta_P,t)=I(P,t)-1\]
for all $t\in\CC$.
\end{proof}

\begin{cor}\label{cor:bound}
Let  $T_{Betti}$ be the set of points $t\in\CC$ where $I(P,t)\ge2$.  Then
\[ \left|T_{tor}\right|\le\left|T_{Betti}\right|\le 2g-2-d+\delta,\]
where $\delta$ is the number of singular fibers of $\pi:\EE\to\CC$.
\end{cor}

\begin{proof}
  From the definitions, $T_{tor}$ is the subset of $T_{Betti}$ where
  $P(t)$ is a torsion point in its fiber, so
  $|T_{tor}|\le|T_{Betti}|$.  Let $S$ be the set of points of $\CC$
  where $\EE$ has bad reduction.  Then $I(P,t)\ge0$ for all $t\in\CC$,
  $I(P,t)\geq1$ for all $t\not\in S$, and $I(P,t)\ge2$ if and only
  $t\in T_{Betti}$.  Thus by the Theorem,
\[2g-2-d\ge\left|T_{Betti}\right|-\left|S\right|\]
and the corollary follows immediately.
\end{proof}

\section{Examples and applications}\label{s:examples}
We consider some explicit families illustrating various aspects of the
main theorem, and we give an application to bounding heights of
integral points on elliptic curves over function fields.

Suppose as usual that $\pi:\EE\to\CC$ is a Jacobian elliptic surface
with zero section $O$, $g$ is the genus of $\CC$,
$d=\deg\left(O^*(\Omega^1_{\EE/\CC})\right)$, and $\delta$ is the
number of singular fibers of $\pi$.

\begin{prop}\label{prop:degen}
If $2g-2-d+\delta<0$, then the group of sections of $\pi$ is finite.
\end{prop}

\begin{proof}
  Suppose $P$ is a section of $\pi$ of infinite order, and let
  $T_{tor}$ be the corresponding set of tangencies between $P$ and
  torsion multisections.  Then by Corollary~\ref{cor:bound} the
  cardinality of $T_{tor}$ would be negative, a contradiction.  Thus,
  there are no sections of infinite order.
\end{proof}

It would be interesting to have a more direct proof of the proposition.
We note that the proposition is sharp in the sense that we give
examples below of elliptic surfaces with $2g-2-d+\delta=0$ and with a
section of infinite order.

\subsection{Degenerate cases}
Next, we give two examples where  $2g-2-d+\delta<0$, one
isotrivial, one non-isotrivial.  In both cases, it is straightforward
to check that the group of sections is torsion, in agreement with
Proposition~\ref{prop:degen}.

Let $E$ be any elliptic curve over $\C$ with a given Weierstrass model
$y^2=x^3+ax+b$ where $a,b\in\C$.  Let $E$ be the twisted elliptic curve
\[E:\quad y^2=x^3+at^2x+bt^3\]
over $\C(t)$, and let $\EE\to\P^1$ be the regular minimal model of
$E/\C(t)$.  Then one verifies easily that $d=1$ and $\delta=2$ ($E$
has $ I_0^*$ reduction at $t=0$ and $t=\infty$ and good reduction
elsewhere), so that $2g-2-d+\delta=-1$.

For a non-isotrivial example, consider 
\[E:\quad y^2=x^3-3t^4(t^2-1)^2x+2t^5(t^2-1)^3\]
over $\C(t)$, and let $\EE\to\P^1$ be the regular minimal model.  Then
one verifies that $d=2$ and $\delta=3$, so $2g-2-d+\delta=-1$.
Moreover, the $j$-invariant of $E$ is $1728t^2/(t^2-1)$, so
$\EE\to\P^1$ is non-isotrivial.  (Thanks to Rick Miranda for pointing
out how to construct an example like this.)
We refer to \cite{Schmickler-Hirzebruch85} and \cite{Nguyen99} for the
complete list of Jacobian elliptic surfaces over $\P^1$ with three
singular fibers.


We have the following general result.
\begin{prop}
  Suppose that $\pi:\EE\to\CC$ is a Jacobian elliptic fibration with
  zero section $O$.  Suppose $\pi$ has everywhere semi-stable
  reduction \textup{(}i.e., the bad fibers are of type $I_b$\textup{)}
  and is non-isotrivial.  Then
\[2g-2-d+\delta>0.\]
Here, as usual, $g$ is the genus of $\CC$,
$d=\deg\left(O^*(\Omega^1_{\EE/\CC})\right)$, and $\delta$ is the
number of singular fibers of $\pi$.
\end{prop}

\begin{proof}
Let $D$ be the divisor
\[D=O+\sum_{\text{bad }t}\pi^{-1}(t)\]
where the sum is over the set of points where the fiber of $\pi$ is
singular.  By hypothesis, each $\pi^{-1}(t)$ appearing in the sum is a
chain of $\P^1$s meeting in nodes, and the divisor $D$ thus has normal
crossings.  

Consider the logarithmic Chern classes $\overline{c}_1(\EE,D)$ and
$\overline{c}_2(\EE,D)$ as in \cite{Urzua11}.  By
\cite[Thm.~9.2]{Urzua11} (which applies since $\pi$ is non-isotrivial),
we have
\[\overline{c}_1(\EE,D)^2<3\overline{c}_2(\EE,D).\]
From the definitions, one computes that
\[\overline{c}_1(\EE,D)^2=4g-4+d+2\delta
\quad\text{and}\quad
\overline{c}_2(\EE,D)=2g-2+\delta,\]
so we find that 
\[0<3\overline{c}_2(\EE,D)-\overline{c}_1(\EE,D)^2=2g-2-d+\delta\]
as desired.
\end{proof}

\subsection{Optimality in the constant case}
Fix an elliptic curve $E$ over $\C$ and a positive integer $g$.  Let
$B=2g-2$ and let $A$ be any integer with $0\le A\le B$.  We will
produce a constant elliptic surface $\EE=\CC\times E$ and a section
whose corresponding $T_{tor}$ satisfies
\[|T_{tor}|=A\le B=2g-2.\]
Indeed, by the Riemann existence theorem, there exists a branched
cover $f:\CC\to\EE$ where $\CC$ is a curve of genus $g$ and $f$ has
simple ramification.  I.e., $f$ is ramified over $2g-2$ points with
distinct images in $\CC$, and the ramification indices are all $2$.
Moreover, we may choose the points in $E$ where $f$ is ramified
freely.

Let $\EE=\CC\times E$ and let $P$ be the section corresponding to the
map $f$.  Then $T_{Betti}$ is exactly the set of branch points, and it
has cardinality $2g-2$.  Moreover, by suitable choice of those points,
we may arrange for $|T_{tor}|$ to take any value between $0$ and
$2g-2$.  This shows that Corollary~\ref{cor:bound} is sharp in the
constant case.

\subsection{Optimality in the non-constant cases}\label{ss:opt}
A similar idea works in the non-constant cases once we have suitable
starting data.  To that end, we construct $\pi:\EE_1\to\P^1$ with
\[2g-2-d+\delta=-2-d+\delta=0,\]
and with a section $P_1$ of infinite order.  By
Corollary~\ref{cor:bound}, there are no tangencies bewteen $P_1$ and a
torsion multisection.

In the isotrivial case, we may take the example considered in
\cite[\S7]{UUpp19}, namely the quotient of the square of an elliptic
curve by the diagonal map $(t,t)\mapsto(-t,-t)$.  The minimal regular
model $\EE_1\to\P^1$ has $d=2$ and $\delta=4$ bad fibers and a section
$P_1$ of infinite order (namely the image of the graph of the identity
map).  The set of tangencies between $P_1$ is thus empty by
Corollary~\ref{cor:bound}.

For a non-isotrivial example, consider the elliptic curve 
\[E_1: \quad y^2=x^3-tx+t\] and let $\EE_1\to\P^1$ be the regular
minimal model.  One computes that $\EE_1$ has
$\deg\left(O^*(\Omega^1_{\EE_1/\P^1})\right)=1$, and good reduction
away from $t=0$, $t=27/4$, and $t=\infty$, and has bad reduction at
these points.  Thus for any non-torsion section, the corresponding set
of torsion tangencies $T_{tor}$ is empty.  Let $P_1$ be the section
corresponding to the rational point $(x,y)=(1,1)$.  Straightforward
calculation shows that $P_1$ is of infinite order (and in fact
generates the group of sections of $\EE_1$).  By
Corollary~\ref{cor:bound}, the corresponding et $T_{tor}$ is empty.

Now fix a positive, even integer $B$ and let $f:\CC\to\P^1$ be a
branched cover with exactly $B$ ramification points and simple
branching over each one.  (We could also insist that $\CC\to\P^1$ have
low degree, say degree 2, but this is not relevant for what follows.)
Let $\EE\to\CC$ be the regular minimal model of the pull back of
$\EE_1\to\P^1$ to $\CC$ (where $\EE_1$ is either of the examples
above), and let $P$ be the section induced by $P_1$.  Assuming that
the branch points of $\CC\to\P^1$ are distinct from the points where
$\EE_1$ has bad reduction, we have $I(P,t)=e_f(t)$ where $e_f(t)$ is
the ramification index of $f$ at $t\in\CC$.

The Riemann-Hurwitz formula yields
\[2g_{\CC}-2-d+\delta=B\]
where $d$ and $\delta$ are the usual invariants attached to $\EE$.
Thus, Corollary~\ref{cor:bound} implies that $|T_{tor}|\le B$.

By choosing some of the branch points of $f$ to be among the points of
$\P^1$ where $P$ takes a torsion value, we may arrange for $|T_{tor}|$
to take any value between 0 and $B$.  This shows that
Corollary~\ref{cor:bound} is sharp.

\subsection{Height bounds}\label{ss:heights}
We consider bounds on heights of integral points over function fields,
as pioneered by Mason \cite[Th.~12, p.~58]{MasonDEoFF}.  Our bound is
a small improvement over that of Hindry and Silverman \cite[Props.~8.2
and 8.3]{HindrySilverman88}.  Although we work over $\C$, the
generalization to any field of characteristic zero is straightforward.

To state the result, let $\CC$ be an irreducible, smooth, projective
curve of genus $g$ over $\C$, let $K$ be the function field
$K=\C(\CC)$, let $E$ be an elliptic curve over $K$, and let
$\hat h:E(K)\to\Q$ be the canonical height function on $E$.  (We give
a precise definition in the proof below.)

Let $S$ be a non-empty, finite set of closed points of $\CC$, and let
$\OO\subset K$ be the subring of functions regular at all points not in
$S$.  Choose an $S$-integral model for $E$, in other words, a
Weierstrass model
\begin{equation}\label{eq:model}
y^2=x^3+Ax+B  
\end{equation}
where $A,B\in\OO$.  Let $\Delta=4A^3+27B^2$ be the discriminant of
this model, and let $T$ be the union of $S$ and the set of points
where $\Delta$ vanishes.  We write $|T|$ for the cardinality of $T$.  

\begin{thm}
  Suppose that $P\in E(K)$ is a non-torsion point that is
  $S$-integral, i.e., whose coordinates $x(P)$ and $y(P)$ in the
  model \eqref{eq:model} are in $\OO$.  Then we have
  \[\hat h(P)\le4g-4+2|T|.\]
\end{thm}

\begin{proof}
  Let $\pi:\EE\to\CC$ be the N\'eron model of $E/K$, and write $O$ and $P$
  for the images in $\EE$ of the zero-section and the section induced
  by $P$ respectively.

  We first recall the definition of $\hat h(P)$ as an intersection
  number following \cite{CoxZucker79} and \cite{Shioda90}.  Associated
  to $P$ there is a unique $\Q$-linear combination of non-identity
  components of fibers of $\pi$ denoted $D_P$ with the property that
  \[\left(P-O+D_P\right).C=0\]
  for every component $C$ of every fiber of $\pi$.  (The dot
  signifies the intersection pairing on $\EE$.)  The canonical height
  is then
  \[\hat h(P) = -\left(P-O+D_P\right).(P-O)\in\Q.\]
  Consulting the table \cite[1.19]{CoxZucker79} reveals that the
  coefficients of $D_P$ are non-negative, and the canonical bundle
  formula for $\EE$ and adjunction show that $O.O=P.P=-d$ where
  $d=\deg(\omega)$ is as in the introduction.  Thus we find
\begin{equation}\label{eq:h-vs-int}
\hat h(P)\le-(P-O).(P-O)=2(P.O)+2d.  
\end{equation}

To finish the proof, we will estimate $P.O$ using
Theorem~\ref{thm:total-int}.  Since $P\neq O$, the intersection number
$P.O$ is a sum of local terms, and we write $(P.O)_t$ for the
contribution at points of intersection lying in $\pi^{-1}(t)$.  If
$t\not\in T$, the model \eqref{eq:model} is minimal and has good
reduction, and $x(P)$ and $y(P)$ are regular (i.e., do not have
poles), so we have $(P.O)_t=0$.  Also, $I(P,t)-1\ge0$ at all points of
good reduction.  Thus
\[\sum_{t\not\in T}(P.O)_t\le\sum_{t\not\in T}\left(I(P,t)-1\right).\]

For any $t$, we have $(P.O)_t\le I(P,t)$, so we also have
\[\sum_{t\in T}(O.P)_t\le\sum_{t\in T}I(P,t)=\sum_{t\in
    T}\left(I(P,t)-1\right)+|T|.\]

Adding the last two displayed equations and applying
Theorem~\ref{thm:total-int} yields
\[P.O\le\sum_{t\in\CC}\left(I(P,t)-1\right)+|T|=2g-2-d+|T|.\]
Using this in Equation~\eqref{eq:h-vs-int} yields the theorem.
\end{proof}

\begin{rem}
  From the proof, we see that the inequality of the theorem is an
  equality if and only if (i) $P$ passes through the identity component
  of every bad fiber (so $D_P=0$ and there are no ``correction
  terms''); and (ii) $P.O=2g-2-d+|T|$.  We give two examples where
  these conditions are satisfied, thus showing that the theorem is
  sharp in these cases.

  For an isotivial example, take the curve $E$ over $\C(t)=\C(\P^1)$ considered in
  \cite[\S7]{UUpp19} and in Section~\ref{ss:opt} above (associated to
  $E_0$ given by $y^2=f(x)=x^3+ax+b$) and the point
  $P_2$ induced by multiplication by 2 on $E_0$.  Let $S=\{\infty\}$.
  Then the model
  \[y^2=x^3+af(t)^2+bf(t)^3\]
  is $S$-integral and the set $T$ consists of $S$ and the three roots
  of $f$ and has cardinality 4.  One checks that $P_2$ is
  $S$-integral, the corresponding section passes through the identity
  component of every fiber, and $P.O=0=2g-2-d+|T|$.

  For a non-isotrivial example, take the other example considered in
  Section~\ref{ss:opt}, namely
  \[E_1: \quad y^2=x^3-tx+t\]
  over $\C(t)$ and let
  \[P=2(1,1)=
    \left(\frac14t^2-\frac32t+\frac14,
      \frac18t^3-\frac98t^2+\frac{15}8t+\frac18\right).\]
If $S=\{\infty\}$, then this model and the point $P$ are $S$-integral
and the corresponding set $T$ is $\{0,27/4,\infty\}$.  Again one finds
that the section corresponding to $P$ passes through the identity
component of every fiber, and $P.O=0=2g-2-d+|T|$.
\end{rem}

\bibliography{database}

\end{document}